\documentclass[reqno]{amsart}

\numberwithin{equation}{section}

  \usepackage[dvips]{graphicx}
  \usepackage{epsfig,psfrag} 
 \usepackage[colorlinks=true]{hyperref}
\hypersetup{urlcolor=blue, citecolor=red}

\usepackage{a4wide}

\newtheorem{maintheorem}{Theorem}

\newtheorem{theorem}{Theorem}[section]
\newtheorem{corollary}{Corollary}
\newtheorem{lemma}[theorem]{Lemma}
\newtheorem{example}{Example}
\newtheorem{proposition}{Proposition}
\newtheorem{conjecture}{Conjecture}
\newtheorem{claim}{Claim}

\theoremstyle{definition}
\newtheorem{definition}[theorem]{Definition}
\newtheorem{remark}{Remark}

\newcommand{\la} {\lambda}

\newcommand{\vfi}{\varphi}

\newcommand{\gera}{{\rm span}}
\newcommand{\Z}{\mathbb{Z}}

\newcommand{\R}{\mathbb{R}}
\newcommand{\T}{\mathbb{T}}
\newcommand{\sS}{\mathbb{S}}

\newcommand{\sing}{\mathrm{Sing}}
\newcommand{\diver}{\mathrm{div}}

\newcommand{\ov}{\overline}

\newcommand{\SG}{{\mathcal G}}

\newcommand{\SL}{{\mathcal L}}

\newcommand{\SO}{{\mathcal O}}

\newcommand{\wt}{\widetilde}

\title[Dominated splittings]
{Dominated Splittings for Flows with singularities}

\author[V. Araujo, A. Arbieto \& L. Salgado]
{Vitor Araujo, Alexander Arbieto and Luciana Salgado}



\thanks{Authors were partially supported by CNPq,
  PRONEX-Dyn.Syst., FAPERJ and FAPESB.}

\address[V.A.]{Instituto de Matem\'atica,
Universidade Federal da Bahia\\
Av. Adhemar de Barros, S/N , Ondina,
40170-110 - Salvador-BA-Brazil}

\email{vitor.d.araujo@ufba.br}

\address[A.A.]{Universidade Federal do Rio de Janeiro
\\
Instituto de Matem\'atica
  \\
  P. O. Box 68530, 21945-970 Rio de Janeiro, Brazil
  }

 \email{arbieto@im.ufrj.br}

 \address[L.S.]{
  Instituto de Matem\'atica Pura e Aplicada -
  Estrada Dona Castorina, 110, Jardim Bot\^anico, 22460-320 Rio de Janeiro, Brazil
}
 \email{lsalgado@impa.br}

\date{\today}

\begin{document}
\maketitle

\begin{abstract}
  We obtain sufficient conditions for an invariant splitting
  over a compact invariant subset of a $C^1$ flow $X_t$ to
  be dominated. In particular, we reduce the requirements to
  obtain sectional hyperbolicity and hyperbolicity.
\end{abstract}

\section{Introduction}
\label{sec:introd}

The theory of hyperbolic dynamical systems is one of the main
paradigm in dynamics. Deve-loped in the 1960s and 1970s after
the work of Smale, Sinai, Ruelle, Bowen
\cite{Sm67,Si68,Bo75,BR75}, among many others, this theory
deals with compact invariant sets $\Lambda$ for
diffeomorphisms and flows of closed finite-dimensional
manifolds having a hyperbolic splitting of the tangent
space. That is, if $X$ is a vector field and $X_t$ is the generated flow then we say that an invariant and compact set $\Lambda$ is hyperbolic if there exists  a
continuous splitting of the tangent bundle over $\Lambda$, $T_\Lambda M= E^s\oplus E^X \oplus E^u$,
where $E^X$ is the direction of the vector field, the
subbundles are invariant under the derivative $DX_t$ of the
flow $X_t$
\begin{align*}
  DX_t\cdot E^*_x=E^*_{X_t(x)},\quad
  x\in\Lambda, \quad t\in\R,\quad *=s,X,u;
\end{align*}
$E^s$ is uniformly contracted by $DX_t$ and $E^u$ is
uniformly expanded: there are $K,\lambda>0$ so that
\begin{align}\label{eq:def-hyperbolic}
\|DX_t\mid_{E^s_x}\|\le K e^{-\lambda t},
  \quad
  \|DX_{-t} \mid_{E^u_x}\|\le K e^{-\lambda t},
  \quad x\in\Lambda, \quad t\in\R.
\end{align}
Very strong properties can be deduced from the existence of
a such structure; see for
instance~\cite{KH95,robinson2004}. 

Weaker notions of hyperbolicity, like the notions of
dominated splitting, partial hyperbolicity, volume
hyperbolicity and singular or sectional hyperbolicity (for
singular flows) have been proposed to try to enlarge the
scope of this theory to classes of systems beyond the
uniformly hyperbolic ones; see~\cite{BDV2004} and
\cite{AraPac2010} for singular or sectional
hyperbolicity.  However, the existence of dominated splittings is the weaker one.

Many researchers have studied the  relations of the existence of dominated splittings
  with other dynamical phenomena, mostly in the discrete time case, such as robust
  transitivity, homoclinic tangencies and heteroclinic
  cycles, and also the possible extension of this notion to
  endomorphisms; see for instance
  \cite{Moryiasu1991,Wen2002,Newhouse2004,MPP04,BDV2004,BGV06,
    LiangLiu2008}.

  However, the notion of dominated splittings deserves
  attention. Several authors used this notion for the Linear
  Poincar\'e flow, see \cite{Do87,Li80,Man82}, and this is
  useful in the absence of singularities, as in
  \cite{GW2006}. We remark that this flow is defined only in
  the set of regular points. However, in the singular case,
  it is a non-trivial question to obtain a dominated
  splitting for the derivative of the flow, and thus
  allowing singularities. Indeed, it is difficult to obtain
  it, from a dominated splitting for the Linear Poincar\'e
  Flow. See \cite{LGW05}, for an attempt to solve this, in
  the context of robustly transitive sets, using the
  extended Linear Poincar\'e Flow.

\begin{definition}\label{def1}
  A \emph{dominated splitting} over a compact invariant set $\Lambda$ of $X$
  is a continuous $DX_t$-invariant splitting $T_{\Lambda}M =
  E \oplus F$ with $E_x \neq \{0\}$, $F_x \neq \{0\}$ for
  every $x \in \Lambda$ and such that there are positive
  constants $K, \lambda$ satisfying
  \begin{align}\label{eq:def-dom-split}
    \|DX_t|_{E_x}\|\cdot\|DX_{-t}|_{F_{X_t(x)}}\|<Ke^{-\la
      t}, \ \textrm{for all} \ x \in \Lambda, \ \textrm{and
      all} \,\,t> 0.
  \end{align}
\end{definition}

The purpose of this article is to study the existence of
dominated splittings for flows with singularities. In one
hand, we study this question in the presence of some
sectional hyperbolicity, once that this theory was build to
understand flows like the Lorenz attractor, which are the
prototype of non-hyperbolic dynamics with singularities, but
with robust dynamical properties. On the other hand, we
present several examples to clarify the role of the
condition on the singularities.

 This article is organized as follows. In
 Section~\ref{sec:statem-results}, we present definitions
 and auxiliary results needed in the proof of
 Theorem~\ref{mthm:domination}, and some examples showing
 that the domination condition on the singularities is
 necessary. In Section~\ref{applic-examp}, we present
 applications of our main result and more examples implying
 that these results are not valid for diffeomorphisms.
 In Section~\ref{sec:pf-sec-lyap} Theorem
 \ref{mthm:sec-lyap-exp} is proved assuming
 Theorem~\ref{mthm:domination}.  In
 Section~\ref{sec:pf-dom-and-hypnocenter}, we prove the
 Theorem~\ref{mthm:domination} and Theorem~
 \ref{mthm:hyperbolic-no-center}. Theorem~\ref{mthm:unifexp3d-hyp}
 and Corollary~\ref{cor:unifexp3d-hyp} are proved in
 Section~\ref{sec:proof-theorem-3d} using the previously
 proved results. Finally, in Section
 \ref{examples}, we present the examples stated in
 Propositions~\ref{t.examples} and \ref{diffeo-case}.

\section{Statements of the results}
\label{sec:statem-results}


Let $M$ be a connected compact finite $n$-dimensional manifold, $n \geq 3$,
with or without boundary. We consider a vector
field $X$, such that $X$ is inwardly transverse to the
boundary $\partial M$, if $\partial M\neq\emptyset$. The flow generated by $X$ is denoted by $\{X_t\}$.

An \emph{invariant set} $\Lambda$ for the flow of $X$ is a
subset of $M$ which satisfies $X_t(\Lambda)=\Lambda$ for all
$t\in\R$.  The \emph{maximal invariant set of the flow} is
$M(X):= \cap_{t \geq 0} X_t(M)$, which is clearly a compact
invariant set.

A \emph{singularity} for the vector field $X$
is a point $\sigma\in M$ such that $X(\sigma)=\vec0$. The
set formed by singularities is denoted by $\sing(X)$.  We say that a \emph{singularity is
  hyperbolic} if the eigenvalues of the derivative
$DX(\sigma)$ of the vector field at the singularity $\sigma$
have nonzero real part.

A compact invariant set $\Lambda$ is said to be
\emph{partially hyperbolic} if it exhibits a dominated
splitting $T_{\Lambda}M = E \oplus F$ such that subbundle
$E$ is \emph{uniformly contracted}, i.e. there exists $C>0$ and $\lambda>0$ such that $\|DX_t|_{E_x}\|\leq Ce^{-\lambda}$ for $t\geq 0$. In this case $F$ is the
\emph{central subbundle} of $\Lambda$.

We say that a $DX_t$-invariant subbundle $F \subset
  T_{\Lambda}M$ is a \emph{sectionally expanding} subbundle
  if $\dim F_x \geq 2$ is constant for $x\in\Lambda$
  and there are positive constants $C , \lambda$ such that for every $x
    \in \Lambda$ and every two-dimensional linear subspace
    $L_x \subset F_x$ one has
    \begin{align}\label{eq:def-sec-exp}
      \vert \det (DX_t \vert_{L_x})\vert > C e^{\la t},
      \textrm{ for all } t>0.
    \end{align}


\begin{definition}\label{sechypset} \cite[Definition
  2.7]{MeMor06} A \emph{sectional-hyperbolic set} is a
  partially hyperbolic set whose singularities are
  hyperbolic and central subbundle is sectionally
  expanding.
\end{definition}

This is a generalization of hyperbolicity for compact
invariant sets with singularities accumulated by regular
orbits, since the Lorenz attractor and higher dimensional
examples are sectional-hyperbolic
\cite{Tu99,BPV97,MeMor06,AraPac2010} and because of the
following result whose proof can be found in \cite{MPP04,BDV2004,AraPac2010}.

\begin{lemma}[Hyperbolic Lemma]
  \label{pr:sec-hyp-no-sing}
  Every compact invariant subset of a sectional-hyperbolic
  set without singularities is a hyperbolic set.
\end{lemma}

The main result of this paper is the following.

\begin{maintheorem}\label{mthm:domination}
  Let $\Lambda$ be a compact invariant set of $X$ such that
  every singularity in this set is hyperbolic. Suppose that
  there exists a continuous $DX_t$-invariant splitting of the tangent bundle
  of $\Lambda$, $T_{\Lambda}M = E \oplus F$,
  where $E$ is uniformly contracted,
  $F$ is sectionally expanding and for some constants
  $C,\lambda > 0$ we have
\begin{align}
  \|DX_t\vert_{E_\sigma}\|\cdot\|DX_{-t}\vert_{F_\sigma}\| &< \label{eq3}
  Ce^{-\lambda t} \quad\text{for all}\quad
  \sigma\in\Lambda\cap\sing(X)\textrm{ and $t\geq 0$}.
\end{align}
Then $T_{\Lambda}M = E \oplus F$ is a dominated splitting.
\end{maintheorem}





As a corollary, we obtain sufficient conditions to obtain
hyperbolicity in the non-singular case using the hyperbolic
lemma.

\begin{corollary}
  \label{mthm:domination-no-sing}
  Let $\Lambda$ be a compact invariant set without
  singularities for the vector field $X$.  Suppose that
  there exists a continuous $DX_t$-invariant decomposition
  of the tangent bundle of $\Lambda$, $T_{\Lambda}M = E
  \oplus F$, where $E$ contracts and $F$ is sectionally
  expanding. Then $\Lambda$ is a hyperbolic set.
\end{corollary}

We note that in Definition \ref{sechypset} domination is
required. As a consequence of Theorem~\ref{mthm:domination},
the domination assumption is only necessary at the
singularities, so that we obtain the following equivalent
definition of sectional-hyperbolicity.

\begin{definition}\label{newdefinition-sechyp}
  A compact invariant set $\Lambda \subset M$ is a
  sectional-hyperbolic set for $X$ if all singularities in
  $\Lambda$ are hyperbolic, there exists a continuous
  $DX_t$-invariant splitting of the tangent bundle on
  $T_{\Lambda}M = E \oplus F$ with constants $C,\lambda > 0$
  such that for every $x \in \Lambda$ and every $t>0$ we
  have
\begin{enumerate}
\item $\Vert DX_t\vert_{E_x} \Vert \leq C e^{- \lambda t}$;
\item $\vert \det DX_t \vert_{L_x}\vert > C^{-1} e^{\la t}$,
  for every two-dimensional linear subspace $L_x \subset
  F_x$;
\item $\|DX_t\vert_{E_\sigma}\|\cdot\|DX_{-t}\vert_{F_\sigma}\|
  < C e^{-\lambda t} $ for all $\sigma\in\Lambda\cap\sing(X)$.
\end{enumerate}
\end{definition}



\subsection{Hyperbolicity versus sectional-expansion}
\label{sec:hyperb-versus-sectio}

We present here some motivation for
Theorem~\ref{mthm:domination}. We observe that, for a
hyperbolic set, both splittings $(E^s\oplus E^X)\oplus E^u$
and $E^s\oplus (E^X\oplus E^u)$ are dominated. Moreover, it
is easy to see that a hyperbolic set $\Lambda$ has no
singularities accumulated by regular orbits within
$\Lambda$. Indeed, any singularity $\sigma$ accumulated by
regular orbits on a hyperbolic set $\Lambda$ would be a
discontinuity point for the hyperbolic splitting, due to the
absence of the flow direction at $\sigma$.

Our next result shows that we cannot have a hyperbolic
splitting without the flow direction, unless we restrict
ourselves to finitely many singularities, all of them
isolated on the nonwandering set.

\begin{theorem}\label{mthm:hyperbolic-no-center}
  Let $\Lambda$ be a compact invariant set of $X$. Suppose
  that there exists a continuous $DX_t$-invariant
  decomposition of the tangent bundle $T_{\Lambda}M = E
  \oplus F$ over $\Lambda$ and constants $C,\lambda > 0,$
  such that for every $x \in \Lambda$ and all $t>0$
\begin{align*}
\Vert DX_t\vert_{E_x}\Vert \leq Ce^{-\lambda t}
\quad\textrm{and}\quad
 \Vert DX_{-t}\vert_{F_{X_t(x)}}\Vert \le Ce^{-\lambda t}.
\end{align*}
Then $\Lambda$ consists of finitely many hyperbolic
singularities.
\end{theorem}

We cannot replace the assumptions on
Theorem~\ref{mthm:hyperbolic-no-center} either by
sectional-expansion or by sectional-contraction, as the
following examples show.

\begin{example}\label{ex:Lorenz-like}
  Consider a Lorenz-like singularity $\sigma$ for a $C^1$
  flow $\{X_t\}_{t\in\R}$ on a $3$-manifold $M$, that is,
  $\sigma$ is a hyperbolic singularity of saddle-type such
  that the eigenvalues of $DX(\sigma)$ are real and satisfy
  $ \lambda_2 < \lambda_3 < 0 < -\lambda_3 < \lambda_1$ and
  $\lambda_1+\lambda_2>0$.

  Let $E_i$ be the eigenspace associated to the eigenvalue
  $\lambda_i$, $i=1,2,3$, and set $E=E_3$ and $F=E_1\oplus
  E_2$. Then the decomposition is trivially continuous, not
  dominated ($F$ admits vectors more sharply contracted than
  those of $E$) but $E$ uniformly contracts lengths of
  vectors and $F$ uniformly expands area, that is, $F$ is
  sectionally-expanded.
\end{example}

\begin{example}
  \label{ex:generalized-Lorenz-like}
  Consider a hyperbolic saddle singularity, $\sigma$, for a
  $C^1$ flow $\{X_t\}_{t\in\R}$ on a $4$-manifold $M$ such
  that the eigenvalues of $DX(\sigma)$ are real and satisfy
  \begin{align*}
    \lambda_2 < \lambda_3 < 0 < \lambda_4 < \lambda_1, \quad
    \lambda_1+\lambda_3>0 \quad\text{and}\quad \lambda_2+\lambda_4<0.
  \end{align*}
  Let $E_i$ be the eigenspace associated to the eigenvalue
  $\lambda_i$, $i=1,2,3,4$, and set $F=E_2\oplus E_4$ and
  $E=E_1\oplus E_3$. Then the decomposition is trivially
  continuous, not dominated as before, $E$ is uniformly area
  contracting, since $DX_{-t}\mid_E$ expands area for $t>0$;
  and $F$ is uniformly area expanding. In other words, $F$
  is sectionally-expanded and $E$ is sectionally-contracted.
\end{example}

In both examples above we have sectional-expansion and
sectional-contraction along the subbundles of a continuous
splitting but the splitting is not dominated.  These
examples involve a trivial invariant set: an equilibrium
point.  But there are examples with compact invariant sets
having singularities accumulated by regular orbits and also
with a dense regular orbit; see Subsection~\ref{examples}.

This suggests that Theorem~\ref{mthm:hyperbolic-no-center}
might be generalized if we assume domination at the
singularities of $\Lambda$ together with sectional-expansion
along $F$ and uniform contraction along $E$ over regular
orbits. This is precisely the content of
Theorem \ref{mthm:domination}.


\section{Applications and other examples}
\label{applic-examp}

To present the next result characterizing
sectional-hyperbolic sets through sectional Lyapunov
exponents, we first have to present more definitions.


We say that a probability measure $\mu$ is $X$-invariant if
$\mu(X_t(U))=\mu(U)$, for every measurable subset $U\subset
M$ and every $t\in \R$. Given a compact $X$-invariant set
$\Lambda$ we say that a subset $Y\subset M$ is a \emph{total
  probability subset} of $\Lambda$ if $\mu(Y)=1$ for every
$X$-invariant measure $\mu$ supported in $\Lambda$. We note
that \emph{each singularity $\sigma$ of $\Lambda$ belongs to $Y$}
since $\mu=\delta_\sigma$ is a $X$-invariant probability
measure.

Let $A:E\times\R\to E$ be a Borel measurable map given by a
collection of linear bijections
\begin{align*}
  A_t(x): E_x\to E_{X_t(x)}, \quad x\in M, t\in\R,
\end{align*}
where $M$ is the base space (we assume it is a manifold) of
the finite dimensional vector bundle $E$, satisfying the
cocycle property
\begin{align*}
  A_0(x)=Id, \quad A_{t+s}(x)=A_t(X_s(x))\circ A_s(x), \quad x\in M, t,s\in\R,
\end{align*}
with $\{X_t\}_{t\in\R}$ a smooth flow over $M$.  We note
that for each fixed $t>0$ the map $A_t: E\to E, v_x\in E_x
\mapsto A_t(x)\cdot v_x\in E_{X_t(x)}$ is an automorphism of
the vector bundle $E$.

The natural example of a linear multiplicative cocycle over
a smooth flow $X_t$ on a manifold is the derivative cocycle
$A_t(x)=DX_t(x)$ on the tangent bundle $TM$ of a finite
dimensional compact manifold $M$. 

A concrete instance of this structure we will use is the
derivative cocycle $A_t(x)=DX_t\mid E_x$ restricted to a
continuous $DX_t$-invariant subbundle $E$ of the tangent
bundle $TM$.

According to the Multiplicative Ergodic Theorem of Oseledets
\cite{arnold-l-1998,BarPes2007}, since in this setting we
have that $\sup_{-1\le t\le1}\log\|A_t(x)\|$ is a bounded
function of $x\in \Lambda$, there exists a subset $R$ of $\Lambda$
with total probability such that for every $x \in R$ there
exists a splitting
 \begin{align}\label{oseledec-split}
   E_x = E_1(x) \oplus \cdots \oplus E_{s(x)}(x)
 \end{align}
 which is $DX_t$-invariant and the following limits, known
 as the \emph{Lyapunov exponents at} $x$, exist
\begin{align*}
\lambda_i(x) = \lim_{t \to \pm\infty} \frac{1}{t} \log \Vert DX_t(x) \cdot v\Vert,
\end{align*}
for every $v \in E_i(x) \setminus \{ 0 \}, i = 1, \dots,
s(x)$. We order these numbers as
$\lambda_1(x)<\dots<\lambda_{s(x)}(x)$. One of these
subbundles is given by the flow direction (at non-singular
points of the flow) and the corresponding Lyapunov exponent
is zero for almost every point.

The functions $s$ and $\lambda_i$ are measurable and
invariant under the flow, i.e., $s(X_t(x)) = s(x)$ and
$\lambda_i (X_t(x)) = \lambda_i (x)$ for all $x\in R$ and
$t\in\R$. The splitting \eqref{oseledec-split} also depends
measurably on the base point $x\in R$. If $F$ is a
measurable subbundle of the tangent bundle then by ``the
Lyapunov exponents of $E$'' we mean the Lyapunov exponents
of the nonzero vectors in $F$.

Given a vector space $E$, we denote by $\wedge^2E$ the
second exterior power of $E$, defined as follows. For a basis
$v_1,\dots, v_n$  of $E$, then $\wedge^2E$ is
generated by $\{v_i\wedge v_j\}_{i\neq j}$. Any linear
transformation $A:E\to F$ induces a transformation
$\wedge^2A:\wedge^2E\to\wedge^2F$. Moreover, $v_i\wedge v_j$
can be viewed as the $2$-plane generated by $v_i$ and $v_j$ if
$i\neq j$; see for instance \cite{arnold-l-1998} for more
information.

In \cite{arbieto2010}, a notion of sectional Lyapunov
exponents was defined and
a characterization of sectional-hyperbolicity
was obtained based on this notion.

\begin{definition}\cite[Definition 2.2]{arbieto2010}
\label{def-sec-lyap-exp}
Given a compact invariant subset $\Lambda$ of $X$ with a
$DX_t$-invariant splitting $T_\Lambda M=E\oplus F$, the
\emph{sectional Lyapunov exponents} of $x$ along $F$ are the
limits
$$\lim_{t\to+\infty}\frac1t
\log\|\wedge^2DX_t(x)\cdot\wt{v}\|$$ whenever they exists,
where $\wt{v}\in \wedge^2F_x-\{0\}$.
\end{definition}

As explained in \cite{arbieto2010}, if
$\{\lambda_i(x)\}_{i=1}^{s(x)}$ are the Lyapunov exponents,
then the sectional Lyapunov exponents at a point $x\in R$
are $\{\lambda_i(x)+\lambda_j(x)\}_{1\leq i<j\leq s(x)}$ and
they represent the asymptotic sectional-expansion in the $F$
direction.

As an application of Theorem~\ref{mthm:domination}, we have
the following result, which is an extension of the main
result in \cite{arbieto2010} assuming continuity of the
splitting, instead of a dominated splitting.


\begin{maintheorem}\label{mthm:sec-lyap-exp}
   Let $\Lambda$ be a compact invariant set for
  a $X$ such that every singularity
  $\sigma\in\Lambda$ is hyperbolic. Suppose that there is a continuous $DX_t$-invariant splitting
  $T_{\Lambda}M=E\oplus F$ such that $T_{\sigma}M = E_{\sigma} \oplus F_{\sigma}$
  is dominated, for every singularity $\sigma\in\Lambda$.

  If the Lyapunov exponents in the $E$ direction are
  negative and the sectional Lyapunov exponents in the $F$
  direction are positive on a set of total probability
  within $\Lambda$, then the splitting is dominated and
  $\Lambda$ is a sectional hyperbolic set.
\end{maintheorem}

We now apply the main result to the setting of weakly
dissipative three-dimensional flows.

We recall that an open subset $U$ is a \emph{trapping
  region} if $\overline{X_t(U)}\subset U, t>0$ and a compact
invariant subset $\Lambda$ is \emph{attracting} if it is the
maximal invariant subset
$\Lambda(U)=\overline{\cap_{t\ge0}X_t(U)}$ inside the
trapping region $U$. We say that a compact invariant subset
$\Lambda$ for the flow of the vector field $X$ is
\emph{weakly dissipative} if $\diver(X)(x)\le0$ for all
$x\in\Lambda$, that is, the flow near $\Lambda$
infinitesimally does not expands volume. 

We also say that a hyperbolic singularity $\sigma$ of $X$ is
\emph{$C^1$-linearizable} if there exists a diffeomorphism
$h:V_\sigma\to B$ from an open neighborhood of $\sigma\in M$
to an open neighborhood $B$ of the origin in $\R^3$ such
that $h(X_t(x))=e^{tA}\cdot h(x), X_s(x)\in V_\sigma$ for
$|s|\le t$ and $A=DX(\sigma)$. For this it is enough that
the spectrum of $DX(\sigma)$ satisfies a finite number of
non-resonance conditions; see e.g. \cite{St58}.  Hence this
is a $C^r$-generic condition for all sufficiently big
$r\ge1$.

\begin{maintheorem}
  \label{mthm:unifexp3d-hyp}
  Let $X$ be a $C^1$ vector field on a three-dimensional
  manifold $M$ admitting a trapping region $U$ whose
  singularities (if any) are hyperbolic. Let us assume that
  the compact invariant subset $\Lambda=\Lambda(U)$  is
  weakly dissipative and endowed with a one-dimensional
  continuous field $F$ of asymptotically backward
  contracting directions, that is, $x\in\Lambda\mapsto F_x$
  is continuous and for each $x\in\Lambda$, $F_x$ is a
  one-dimensional subspace of $T_xM$, and also
    \begin{align}\label{eq:posLyap}
      \liminf_{t\to+\infty}\frac1t\log\|DX_{-t}\mid F_x\|<0
      \quad\text{for all}\quad x\in\Lambda.
    \end{align}
    If at each $\sigma\in U$ there exists a complementary
    $DX_t$-invariant direction $E_\sigma$ such that
    $E_\sigma\oplus F_\sigma=T_\sigma M$ is a dominated
    splitting, then $\Lambda$ is a hyperbolic set (in
    particular, $\Lambda$ has no singularities).
\end{maintheorem}

Since we use the attracting and dominated splitting
assumption only to prove the non-existence of singularities
in $\Lambda$, as a consequence of the proof we obtain the
following.

\begin{corollary}
  \label{cor:unifexp3d-hyp}
  Let $X$ be a $C^1$ vector field on a three-dimensional
  manifold $M$ admitting 
  compact invariant subset $\Lambda$, without singularities,
  which is weakly dissipative and endowed with a
  one-dimensional continuous field $F$ of asymptotically
  backward contracting directions.

  Then $\Lambda$ is a hyperbolic set.
\end{corollary}

These results are extremely weak versions of the following
conjecture of Viana, presented in~\cite{Vi98}. 
\begin{conjecture}\label{conj:viana}
  If an attracting set $\Lambda(U)$ of smooth map/flow has a
  non-zero Lyapunov exponent at Lebesgue almost every point
  of its isolating neighborhood $U$, that is 
    \begin{align}\label{eq:posLyapLebqtp}
      \limsup_{t\to+\infty}\frac1t\log\|DX_t\|>0
      \quad\text{for Lebesgue almost every}\quad x\in U,
    \end{align}
    then $\Lambda(U)$ has a \emph{physical measure}: there
    exists an invariant probability measure $\mu$ supported
    in $\Lambda(U)$ such that for all continuous functions
    $\vfi:U\to\R$
    \begin{align*}
      \lim_{t\to+\infty}\frac1t\int_0^t\vfi(X_t(x))\,dt=\int\vfi\,d\mu
      \quad\text{for Lebesgue almost every}\quad x\in U.
    \end{align*} 
\end{conjecture}
Indeed, if $U$ is a trapping region and $\Lambda=\Lambda(U)$
satisfies the assumptions of
Theorem~\ref{mthm:unifexp3d-hyp}, then $\Lambda$ is a
hyperbolic attracting set, for which it is well-known that
(\ref{eq:posLyapLebqtp}) is true and there exists some
physical measure; see e.g. \cite{BR75,Bo75,BDV2004}.


%



We now provide several examples which clarify the role of the
assumption of domination at the singularities in
Theorem~\ref{mthm:domination} and the relations between
dominated splittings for diffeomorphisms and flows.

\begin{proposition}
\label{t.examples}
There exist the following examples of vector fields and
flows.
\begin{enumerate}
\item A vector field with an invariant and compact set
  containing singularities accumulated by regular orbits
  inside the set, with a non-dominated and continuous
  splitting of the tangent space, but satisfying uniform
  contraction and sectional-expansion.
\item A vector field with an invariant and compact set with
  a continuous and invariant splitting $E\oplus F$ such that
  $E$ is uniformly contracted , $F$ is area expanding but it
  is not dominated.
\item A suspension flow whose base map has a dominated splitting but the
flow does not admit any dominated splitting.
\end{enumerate}
\end{proposition}

Item $(3)$ above is based on a diffeomorphism described in
\cite{BoV00} and similar to another suggested by Pujals in
\cite[Example B.12]{BDV2004}.

An important remark is that our results are not valid for
diffeomorphisms, as the following result shows. 

\begin{proposition}
\label{diffeo-case}
There exists a transitive hyperbolic set for a
diffeomorphism with a non-dominated splitting which is
sectionally-expanding and sectionally-contracting.
\end{proposition}


\section{Proof of Theorem \ref{mthm:sec-lyap-exp}}
\label{sec:pf-sec-lyap}

We follows the lines of
\cite{arbieto2010}.  The following proposition, whose proof
can be found in \cite{arbieto2010}, is the main auxiliary
result in the proof.

We fix a compact $X_t$-invariant subset $\Lambda$.  We say
that a family of functions $\{f_t:\Lambda\to \R\}_{t\in \R}$
is sub-additive if for every $x\in M$ and $t,s\in \R$ we
have that $f_{t+s}(x)\leq f_s(x)+f_t(X_s(x))$. The
Subadittive Ergodic Theorem (see e.g. \cite{Wa82}) shows
that the function
$\overline{f}(x)=\liminf\limits_{t\to+\infty}\frac{f_t(x)}{t}$
coincides with
$\wt{f}(x)=\lim\limits_{t\to+\infty}\frac{1}{t}f_t(x)$ in a
set of total probability in $\Lambda$.

\begin{proposition}
\label{prop3.4-arbieto}
Let $\{t\mapsto f_t:\Lambda\to \R\}_{t\in \R}$ be a
continuous family of continuous function which is
subadditive and suppose that $\ov{f}(x)<0$ in a set of
total probability. Then there exist constants $C>0$ and
$\lambda<0$ such that for every $x\in \Lambda$ and every
$t>0$ we have
$\exp(f_t(x))\leq C \exp(\frac{\lambda t}{2}).$
\end{proposition}

\begin{proof}
  See \cite[Proposition 3.4]{arbieto2010}.
\end{proof}

\begin{remark}\label{rmk:Arbieto}
  In \cite{arbieto2010} only $\Lambda = M(X)$ was considered
  to conclude that $\{X_t\}_{t\in\R}$ is a sectional-Anosov
  flow. However all statements are valid for a compact
  invariant subset of $M$.
\end{remark}

\begin{proof}[Proof of Theorem~\ref{mthm:sec-lyap-exp}]
  Define the following families of continuous functions on
  $\Lambda$
\begin{align*}
  \phi_t(x)=\log \| DX_t|_{E_x}\| \quad\textrm{and}\quad
  \psi_t(x)=\log \|\wedge^2 DX_{-t}|_{F_{X_t(x)}}\|,
  \quad (x,t)\in\Lambda\times\R.
\end{align*}
Both families $\phi_t,\psi_t:\Lambda\to\R$ are easily seen
to be subadditive.  From Proposition \ref{prop3.4-arbieto}
applied to $\phi_t(x)$ and the hypothesis on the Lyapunov
exponents, there are constants $C> 0$ and $\gamma < 0$ such
that $ \exp(\phi_t(x))=\|DX_t|_{E_x}\|\leq C\exp(\gamma t)$
for all $t>0$ showing that $E$ is a contractive subbundle.

Analogously, the hypothesis on the sectional Lyapunov
exponents and Proposition \ref{prop3.4-arbieto} applied to
the function $\psi_t(x)$ provides constants $D > 0$ and
$\eta < 0$ for which $\|\wedge^2DX_{-t}|_{F_{X_t(x)}}\|\le
De^{\eta t}$, so $F$ is a sectionally expanding subbundle.
Now Theorem \ref{mthm:domination} ensures that the splitting
$E \oplus F$ is a dominated splitting.
\end{proof}


\section{Proof of Theorem \ref{mthm:domination} and Theorem \ref{mthm:hyperbolic-no-center}}
\label{sec:pf-dom-and-hypnocenter}

We need an auxiliary lemma for which the following
definition is necessary.  Given a pair of subspaces $E_x,
F_x$ of $T_xM$ such that $E_x\cap F_x=\{\vec0\}$, the angle
$\angle(E_x,F_x)$ between the two subspaces is defined by
\begin{align*}
  \sin \angle(E_x,F_x) := \frac1{\|\pi(E_x)\|}, \quad
  x\in\Lambda,
\end{align*}
where $\pi(E_x):E_x\oplus F_x\to E_x$ is the projection onto
$E_x$ parallel to $F_x$ defined in the vector space
$E_x\oplus F_x$.

We say that an invariant splitting $E\oplus F=T_\Lambda M$
of the tangent bundle over a invariant subset $\Lambda$
\emph{has angle uniformly bounded away from zero} if the
dimensions of the fibers $E_x$, $F_x$ are constant for all
$x$ in $\Lambda$ and there exists $\theta_0>0$ such that
$\sin\angle(E_x,F_x)\ge\theta_0$ for each $x$ in $\Lambda$.

The next lemma specifies the subbundle which contains the
flow direction.

\begin{lemma}
  \label{le:flow-center}
  Let $\Lambda$ be a compact invariant set for $X$.  Given
  an invariant splitting $E\oplus F$ of $T_\Lambda M$ with
  angle bounded away from zero over $\Lambda$, such that $E$
  is uniformly contracted, then the flow direction is
  contained in the $F$ subbundle, for all $x\in \Lambda$.
\end{lemma}

\begin{proof}
  We denote by $\pi(E_x):T_x M\to E_x$ the projection on
  $E_x$ parallel to $F_x$ at $T_x M$, and likewise
  $\pi(F_x):T_x M\to F_x$ is the projection on $F_x$
  parallel to $E_x$. We note that for $x\in\Lambda$
  \begin{align*}
    X(x)=\pi(E_x)\cdot X(x) + \pi(F_x)\cdot X(x)
  \end{align*}
  and for $t\in\R$, by linearity of $DX_t$ and
  $DX_t$-invariance of the splitting $ E\oplus F$
  \begin{align*}
    DX_{t}\cdot X(x)
    &=
    DX_{t}\cdot\pi(E_x)\cdot X(x)
    +
    DX_{t}\cdot\pi(F_x)\cdot X(x)
    \\
    &=
    \pi(E_{X_{t}(x)})\cdot DX_{t} \cdot X(x)
    + \pi(F_{X_{t}(x)}) \cdot DX_{t} \cdot X(x)
  \end{align*}
  Let $z$ be a limit point of the negative orbit of $x$.
  That is, we assume without loss of generality since
  $\Lambda$ is compact, that there is a strictly increasing
  sequence $t_n\to+\infty$ such that
  $\lim\limits_{n\to+\infty}x_{n}:=\lim\limits_{n\to+\infty} X_{-t_n}(x) = z$. Then
  $z\in\Lambda$ and, if $\pi(E_x)\cdot X(x)$ is
    not the zero vector, we get
  \begin{align}
   \lim\limits_{n\to+\infty} DX_{-t_n}\cdot X(x)
    &=
    \lim\limits_{n\to+\infty} X(x_n) = X(z), \quad\text{but
      also}
    \nonumber
    \\
   \lim\limits_{n\to+\infty} \|DX_{-t_n}\cdot\pi(E_x)\cdot X(x)\|
    &\ge
   \lim\limits_{n\to+\infty} c^{-1} e^{\lambda t_n}\|\pi(E_x)\cdot
    X(x)\| = +\infty.
    \label{eq:expand}
  \end{align}
  This is possible only if the angle between
  $E_{x_n}$ and $F_{x_n}$ tends to zero when
  $n\to+\infty$.

  Indeed, using the Riemannian metric on $T_y M$, the angle
  $\alpha(y)=\alpha(E_y,F_y)$ between $E_y$ and $F_y$ is related to
  the norm of $\pi(E_y)$ as follows:
  $\|\pi(E_y)\|=1/\sin(\alpha(y))$. Therefore
  \begin{align*}
    \|DX_{-t_n}\cdot\pi(E_x)\cdot X(x)\|
    &=
    \|\pi(E_{x_n})\cdot DX_{-t_n} \cdot X(x) \|
    \\
    &\le
    \frac1{\sin(\alpha(x_n))}
    \cdot
    \|DX_{-t_n} \cdot X(x) \|
    \\
    &=
    \frac1{\sin(\alpha(x_n))}
    \cdot \|X(x_n)\|
  \end{align*}
  for all $n\ge1$. Hence, if the sequence
  $\|DX_{-t_n}\cdot\pi(E_x)\cdot X(x)\|$ in the left hand
  side of \eqref{eq:expand} is unbounded, then
  $\lim\limits_{n\to+\infty}\alpha(X_{-t_n}(x)) = 0$.

  However, since the splitting $E\oplus F$
  has angle bounded away from zero over the
  compact $\Lambda$, we have obtained a contradiction. This
  contradiction shows that $\pi(E_x)\cdot X(x)$ is always
  the zero vector and so $X(x)\in F_x$ for all
  $x\in\Lambda$.
\end{proof}

The proof of Theorem~\ref{mthm:hyperbolic-no-center} follows
from Lemma~\ref{le:flow-center}.

\begin{proof}[Proof of
  Theorem~\ref{mthm:hyperbolic-no-center}]
  We assume, arguing by contradiction, that there exists a
  point $x \in \Lambda \setminus\sing(X)$, that is, $x$ is a
  regular point: $X(x)\neq\vec0$.  Hence
  Lemma~\ref{le:flow-center} ensures that $X(x)\in F_x$. But
  the same Lemma~\ref{le:flow-center} applied to the
  reversed flow $X_{-t}$ generated by the field $-X$ on the
  same invariant set $\Lambda$ shows that $X(x)\in
  E_x$. Thus $X(x)\in E_x\cap F_x=\{\vec0\}$. This
  contradiction ensures that $\Lambda\subset\sing(X)$. But
  our assumptions on the splitting show that each $\sigma\in
  \Lambda$ is a hyperbolic singularity: the eigenvalues of
  $DX\mid_{E_\sigma}$ are negative and the eigenvalues of
  $DX\mid_{F_\sigma}$ are positive. It is well-known that
  hyperbolic singularities $\sigma$ are isolated in the
  ambient manifold $M$; see e.g. \cite{PM82}. We conclude
  that the compact $\Lambda$ is a finite set of hyperbolic
  singularities.
\end{proof}


Before proving Theorem~\ref{mthm:domination} we need some
extra notions and definitions.

We recall that, for a invertible continuous linear operator
$L$, the minimal norm (or co-norm) is equal to $ m(L) := \|
L^{-1} \|^{-1}$. Hence condition \eqref{eq:def-dom-split}
in Definition~\ref{def1} is equivalent to
\begin{align}\label{eq:domequiv}
  \|DX_t|_{E_x}\| < Ke^{-\la t}\cdot m(DX_{t}|_{F_x}) \textrm{
    for all} \ x \in \Lambda \textrm{ and for all } t> 0.
\end{align}

The Multiplicative Ergodic Theorem of Oseledets
\cite{arnold-l-1998,BarPes2007} provides the following
properties of Lyapunov exponents and Lyapunov subspaces, in
addition to the existence of Lyapunov exponents on the total
probability subset $R$ of $\Lambda$, as presented in
Section~\ref{applic-examp}. For any pair of disjoint subsets
$I,J\subset\{1,\dots, s(x)\}$, the angle between the bundles
$E_I(x)=\bigoplus_{i\in I}E_i(x)$ and
$E_J(x)=\bigoplus_{j\in J} E_j(x)$ decreases at most
subexponentially along the orbit of $x$, that is
\begin{align*}
  \lim_{t\to\pm\infty}\frac1t\log\sin\angle(E_I(X_t(x)),E_J(X_t(x)))=0,
  \quad x\in R;
\end{align*}
which implies, in particular, that for any pair
$i,j\in\{1,\dots,s(x)\}$ with $i\neq j$ and $v_i\in
E_i(x)\setminus\{\vec0\}, v_j\in
E_j(x)\setminus\{\vec0\}$
\begin{align*}
  \lim_{t\to\pm\infty}\frac1t\log|\det(DX_t\mid \gera\{v_i,v_j\})|
  =
  \lambda_i(x)+\lambda_j(x).
\end{align*}


\begin{proof}[Proof of Theorem~\ref{mthm:domination}]
  Let $x\in R\subset\Lambda$ be a regular point of the flow
  of $X$. From Lemma~\ref{le:flow-center} we know that
  $X(x)\in F_x$. Since $F$ is a $DX_t$-invariant subbundle,
  considering the linear multiplicative cocycle
  $A_t(z)=DX_t\mid F_z$ for $z\in\Lambda, t\in\R$, there
  exists a splitting $F_x=\bigoplus_{j=1}^{s(x)}F_j(x)$ of
  $F_x$ into a direct sum of Lyapunov subspaces. One of
  these subspaces is $E^X$ generated by $X(x)\neq\vec0$,
  which we rename $F_1(x)=E^X_x$ in what follows. We also
  have the corresponding Lyapunov exponents
  $\lambda_j^F(x),j=1,\dots, s(x)$.

  Fixing $i=2,\dots, s(x)$ and $v\in
  F_i(x)\setminus\{\vec0\}$, we consider $\gera\{X(x),v\}$.
  From the assumption of sectional expansion of area
  together with the subexponential control on angles at
  regular points we obtain
  \begin{align*}
    0<\lambda\le\liminf_{t\to+\infty}\frac1t\log|\det
    (DX_t\mid\gera\{X(x),v\})|=\lambda^F_1(x)+\lambda^F_i(x)=\lambda^F_i(x).
  \end{align*}
Hence $\lambda^F_i(x)\ge\lambda>0$ for all $i=2,\dots,s(x),
x\in R$ and $X(x)\neq\vec0$.

If we consider the cocycle $B_t(z)=DX_t\mid E_z,
x\in\Lambda, t\in\R$, then we find the splitting
$E_x=\bigoplus_{i=1}^{r(x)}E_j(x)$ into Lyapunov subspaces
and the corresponding Lyapunov exponents
$\lambda_i^E(x),i=1,\dots,r(x)$. It is easy to see that the
assumption of uniform contraction along $E$ implies that
$\lambda_i^E(x)\le-\lambda<0$ for all $x\in R$ and $1\le
i\le r(x)$.

Then $\phi_t(x)=\log \frac{\| DX_t|{E_x}\|}{m(DX_t\mid
  F_x)}$ is a subadditive family of continuous functions
satisfying
\begin{align*}
  \overline{\phi}(x)
  &=
  \liminf_{t\to+\infty}\frac{\phi_t(x)}{t}\le
  \liminf_{n\to+\infty}\frac1t\log\|
  DX_t|{E_x}\|
  -\limsup_{n\to+\infty}\frac1t\log m(
  DX_t|{F_x})
  \\
  &=
  \max\{\lambda_i^E(x),1\le i\le r(x)\}
  -
  \min\{\lambda_i^F(x),1\le i\le s(x)\}
  \le
  -\lambda-0 = -\lambda
\end{align*}
for all $x\in R$ such that $X(x)\neq\vec0$.

For $x=\sigma$ for some singularity $\sigma\in\Lambda$ we
have $\overline\phi(\sigma)\le-\lambda$ as a direct
consequence of the assumption of domination at the
singularities.

We have shown that $\overline\phi(x)<0$ for all $x\in
R$. Applying Proposition~\ref{prop3.4-arbieto} we see that
there exists $C>0$ such that $\exp \phi_t(x)\le Ce^{-\lambda
  t/2}$ for all $x\in\Lambda$ and $t>0$, which gives us the
condition (\ref{eq:domequiv}). This concludes the proof of
Theorem~\ref{mthm:domination}.
\end{proof}

The proof of Corollary~\ref{mthm:domination-no-sing} follows
from Theorem~\ref{mthm:domination} and the Hyperbolic
Lemma~\ref{pr:sec-hyp-no-sing}.


\section{Proof of Theorem~\ref{mthm:unifexp3d-hyp} and
  Corollary~\ref{cor:unifexp3d-hyp}}
\label{sec:proof-theorem-3d}

We start by showing that the one-dimensional asymptotically
backward contracting subbundle $F$ over $\Lambda$ is in fact
a uniformly expanding subbundle. Indeed, we note that
$f_t(x)=\log\|DX_{-t}\mid F_x\|$ is a subadditive family of
continuous functions satisfying $\overline{f}(x)<0$ for all
$x\in\Lambda$. In particular, $\overline{f}(x)<0$ in a set
of total probability. Hence by
Proposition~\ref{prop3.4-arbieto} we obtain $f_t(x)\le
Ce^{-t\lambda/2}$ for all $t>0$ and some constants
$\lambda,C>0$. (This trivially implies $\|DX_t\mid
F_{X_t(x)}\|\ge e^{t\lambda/2}/C, x\in\Lambda, t\ge0$ and so
$\|DX_T\mid F_x\|>2$ for all $x\in\Lambda$ and $T\ge
2\log(2C)/\lambda$.)

Now we use the assumption that $\Lambda$ is an attracting
set whose possible singularities $\sigma$ are hyperbolic,
$C^1$ linearizable and admit a complementary
$DX_t$-invariant direction $E_\sigma$ such that
$E_\sigma\oplus F_\sigma=T_\sigma M$ is a dominated
splitting.

\begin{lemma}\label{le:nosing}
  The attracting set $\Lambda$ does not contain singularities.
\end{lemma}

\begin{proof}
  By the $C^1$-linearization assumption and the
  Stable/Unstable Manifold Theorems, there exists a
  corresponding $X_t$-invariant (weak- or strong-)unstable
  one-dimensional manifold $W^u_\sigma$ contained in
  $\Lambda$ (since $\Lambda$ is attracting) such that
  $F_\sigma=T_\sigma W^u_\sigma$ and
  $X_{-t}(z)\xrightarrow[t\to+\infty]{}\sigma$ for each
  $z\in W^u_\sigma$.

  \begin{claim}\label{cl:XinF}
    For all $z\in W^u_\sigma$ we have $X(z)\in F_z$.
  \end{claim}

  We note that this is trivially true for $z=\sigma$.  Let
  us assume, by contradiction, that $X(z)\notin F_z$ for
  some $z\in W^u_\sigma\setminus\{\sigma\}$ (and thus for
  all such $z$ by $DX_t$-invariance). By assumption there
  exists a diffeomorphism $h:V_\sigma\to B$ from an open
  neighborhood of $\sigma\in M$ to an open neighborhood $B$
  of the origin in $\R^3$ such that $h(X_t(x))=e^{tA}\cdot
  h(x), X_s(x)\in V_\sigma$ for $|s|\le t$ and
  $A=DX(\sigma)$. By a linear change of coordinates we can
  assume without loss of generality that $Dh(\sigma)\cdot
  F_\sigma=\R\times0^2$ and $Dh(\sigma)\cdot
  E_\sigma=0\times\R^2$.  Let $v:V_\sigma\cap\Lambda\to\R^3$
  denote the continuous unitary vector field such that
  $\gera(v_x)=Dh(x)\cdot F_x, x\in V_\sigma\cap\Lambda$.

  We are assuming that $v_z$ has a non-zero component along
  the $E_\sigma\approx0\times\R^2$ direction for $z\in
  W^u_\sigma\setminus\{\sigma\}$.  Since
  $X_{-t}(z)\xrightarrow[t\to+\infty]{}\sigma$ the
  domination of the splitting $E_\sigma\oplus F_\sigma$
  ensures that $\angle(e^{tA}\cdot
  v_z,E_\sigma)\xrightarrow[]{t\to-\infty}0$. However the
  invariance of $F$ ensures that
  $F_{X_{-t}(z)}\xrightarrow[]{t\to-\infty}F_\sigma$ and
  since $h$ is a $C^1$ conjugation we also have 
  $\angle(e^{tA}\cdot v_z,F_\sigma)\xrightarrow[]{t\to-\infty}0$.
  This contradiction proves the claim.

Using the claim together with the previous estimate on
expansion along $F$, we have that
\begin{align*}
  \infty>\sup_{x\in\Lambda}\|X(x)\|
  \ge 
  \|X(X_{kT}(z))\|
  =
  \|DX_{kT}\cdot X(z)\|\ge2^k\|X(z)\|, k\ge1
\end{align*}
which is a contradiction and completes the proof of the lemma.
\end{proof}

From now we use only the assumptions of
Corollary~\ref{cor:unifexp3d-hyp}, that is, that $\Lambda$
is a weakly dissipative compact invariant subset without
singularities having a continuous field of asymptotically
backward contracting directions.

In this setting, the Linear Poincar\'e Flow is well-defined
over $\Lambda$. If we denote by $\SO_x:T_xM\to N_x$ the
orthogonal projection from $T_xM$ to $N_x=\{v\in T_xM:
<v,X(x)>=0\}$ the orthogonal complement of the field
direction at $x\in\Lambda$, then the \emph{Linear Poincar\'e
  Flow} is given by $P^t_x=\SO_{X_t(x)}\circ DX_t:T_xM\to
N_{X_t(x)}, x\in\Lambda, t\in\R$.

We note that since the flow direction is invariant, from the
splitting $T_xM=E^X_x\oplus N_X$ we can write $DX_{t}=
\begin{pmatrix}
  \frac{\|X(X_{-t}(x))\|}{\|X(x)\|} & \star \\ 0 & P^{-t}
\end{pmatrix}$ and since $|\det DX_{-t}|\ge1$ by the weak
dissipative assumption, we get $|\det(P^{-t})| \ge
m_0=\min_{z\in\Lambda}\frac{\|X(x)\|}{\|X(X_{-t}(x))\|}>0$.

\begin{claim}\label{cl:FLPcontract}
  The subbundle $\tilde F$ of the normal bundle given by
  $\{\tilde F_x=\SO_x(F_x)\}_{x\in\Lambda}$ is
  $P^t$-invariant and uniformly backward contracting.
\end{claim}

Indeed, the continuity of the subbundle $F$ ensures that
$\angle(F_x, X(x))$ is bounded away from zero, and so there
exists $\kappa>0$ such that for every $v\in \tilde F_x$ we
have $|\beta|\le\kappa\|v\|$ such that $\beta X(x)+v=u\in
F_x$. This ensures that
\begin{align*}
  \|P^{-t}_x\cdot v\|
  &=
  \|\SO_{X_{-t}(x)}( DX_{-t}\cdot (u-\beta X(x)))\|
  \le
  \|\SO_{X_{-t}(x)}(DX_{-t}\cdot u -\beta X(X_{-t}(x)))\|
  \\
  &\le
  \|DX_{-t}\cdot u\|
  \le 
  C e^{-t\lambda/2}\|v+\beta X(x)\|
  \le 
  C e^{-t\lambda/2}(\|v\|+\kappa\|v\|\| X(x)\|)
  \\
  &\le
  C\left(1+\kappa\sup_{z\in\Lambda}\|X(z)\|\right)e^{-t\lambda/2}\|v\|
  =
  \tilde C e^{-t\lambda/2}\|v\|.
\end{align*}
This proves Claim~\ref{cl:FLPcontract}.

This is enough to guarantee the existence of a complementary
$P^{-t}$-invariant and expanding subbundle through a graph
transform technique. To this end,
let us consider the one-dimensional subbundle $G_x$
of $N_x$ orthogonal to $\tilde F_x$. We can then write
$N_x=\tilde F_x\oplus G_x$ and so, since $\tilde F$ is
$P^t$-invariant, we get $P^{-t}_x=
\begin{pmatrix}
  a_t(x)  &  b_t(x) \\ 0 & c_t(x)
\end{pmatrix}$, where $a_t(x):\tilde F_x\to \tilde F_{X_{-t}(x)},
c_t(x):G_x\to G_{X_{-t}(x)}$ and $b_t(x):G_x\to F_{X_{-t}(x)}$.
We know that $|a_t(x)|\le\tilde C e^{-t\lambda/2}$ and since
$0<m_0\le|\det(P^{-t})|= |a_t(x)|\cdot|c_t(x)|$ we obtain
$|c_t(x)|\ge m_0 e^{t\lambda/2}/\tilde C, x\in\Lambda,
t>0$. We fix $T>0$ such that $m_0 e^{T\lambda/2}/\tilde C\ge2$.

Let us now consider the family of continuous one-dimensional
complementary subspaces of $\tilde F_x$ inside $N_x$ given
by the graph of a linear map $L:G_x\to \tilde F_x$, i.e., we
consider the vector space $\SG$ of all $\SL=\{L_x:G_x\to
\tilde F_x: x\in \Lambda\}$ with the norm
$\|\SL^1-\SL^2\|=\sup_{x\in\Lambda} \|L^1_x-L^2_x\|$. We
have that $\SG$ with this norm is a Banach space since
$\Lambda$ is compact.

It is easy to see that the map $P^{-t}$ transforms the graph
of $L_x$ into the graph of $\hat L_{X_{-t}(x)}=
(P^{-t}_x\cdot L_x+\pi_{X_{-t}(x)}\cdot P^{-t})\circ
[(I-\pi_{X_{-t}(x)})\cdot P^{-t}]^{-1}$, where $\pi_z:N_z\to
\tilde F_z$ is the orthogonal projection from $N_z$ into
$\tilde F_z$ at $z\in\Lambda$. In this way we get a map
$P^{-T}:\SG\to\SG$ which we claim is a contraction with
respect to the norm given above. Indeed, for any pair
$\SL^i=\{L_x^i\}_{x\in\Lambda}, i=1,2$ we get
\begin{align*}
  \|\hat L^1_{X_{-T}(x)}-\hat L^2_{X_{-T}(x)}\|
  &=
  \|P^{-t}_x\cdot(L^1_x-L^2_x)\circ
  [(I-\pi_{X_{-t}(x)})\cdot P^{-t}]^{-1}\|
  \\
  &\le
  |a_T(x)|\cdot\|L^1_x-L^2_x\|/|c_t(x)|
  \le
  \frac12\|L^1_x-L^2_x\|
\end{align*}
for each $x\in\Lambda$. We thus have a $P^{-T}$ fixed
point $\SL^0$ defining a $P^{-T}$-invariant
subbundle $\tilde E=\{\tilde E_x=\{(u,L^0_x(u)):u\in
G_x\}_{x\in\Lambda}$ of the normal bundle. Since $P^{-T}$
commutes with each $P^s$ for all $s\in\R$, we see that
$\tilde E$ is $P^s$-invariant for all $s\in\R$.

\begin{claim}\label{cl:Einv}
  The subbundle $E=\tilde E\oplus E^X$ is $DX_t$-invariant.
\end{claim}
Indeed, for any $v=u+\beta X(x)$
with $u\in\tilde E_x,\beta\in\R$, we have $DX_t\cdot
v=DX_t\cdot u+\beta X(X_t(x))$ and $DX_t\cdot u=P^t_x\cdot
u+\gamma X(X_t(x))$ for some $\gamma\in\R$, thus $DX_t\cdot
v=P^t_x\cdot u+(\beta+\gamma) X(X_t(x))$ with $P^t_x\cdot
u\in \tilde E_{X_t(x}$, hence $DX_t\cdot v\in E_{X_t(x)}$
and the claim is proved.

At this point we have a continuous splitting $T_xM=E_x\oplus
F_x, x\in\Lambda$. Using this invariant splitting we can
write $DX_{-t}=\begin{pmatrix} A_t(x) & 0 \\ 0 & B_t(x)
\end{pmatrix}$ and the continuity of the splitting ensures
that $\angle(E_x,F_x)$ is bounded away from zero. This
ensures that $1\le|\det DX_{-t}|=|\det A_t(x)|\cdot
|B_t(x)|\cdot \sin\angle(E_x,F_x)$ and since
$B_t(x)=DX_{-t}\mid F_x$ we have $|B_t(x)|\le C
e^{-t\lambda/2}$, hence $|\det A_t(x)|$ grows exponentially
fast.

This shows that $\Lambda$ has a continuous invariant
splitting $E\oplus F$ with $F$ uniformly contracting and $E$
area expanding for the backward flow. In a three-dimensional
manifold this shows that $E$ is sectionally-expanding for
$-X$. From Corollary~\ref{mthm:domination-no-sing} we have
that $\Lambda$ is a hyperbolic set for $-X$, hence $\Lambda$
is a hyperbolic set and the proof of
Theorem~\ref{mthm:unifexp3d-hyp} and
Corollary~\ref{cor:unifexp3d-hyp} is complete.


\section{Examples}\label{examples}

This section is devoted to present some examples.

\subsection{Proof of Proposition~\ref{t.examples}}

Here we present examples stated in Proposition~\ref{t.examples}.

\vspace{0.2in}

{\bf Item $(1)$:}
  \label{ex:morales}
  Consider a vector field on the plane $\R^2$ with a double
  saddle homoclinic connection which expands volume and
  multiply this vector field by a contraction along the
  vertical direction; see
  Figure~\ref{fig:double-homocl-saddle}. For the
  construction of the planar vector field, see
  e.g. \cite[Chapter 4, Section 9]{perko01}.

  \begin{figure}[htpb]
    \centering
    \includegraphics[width=9cm]{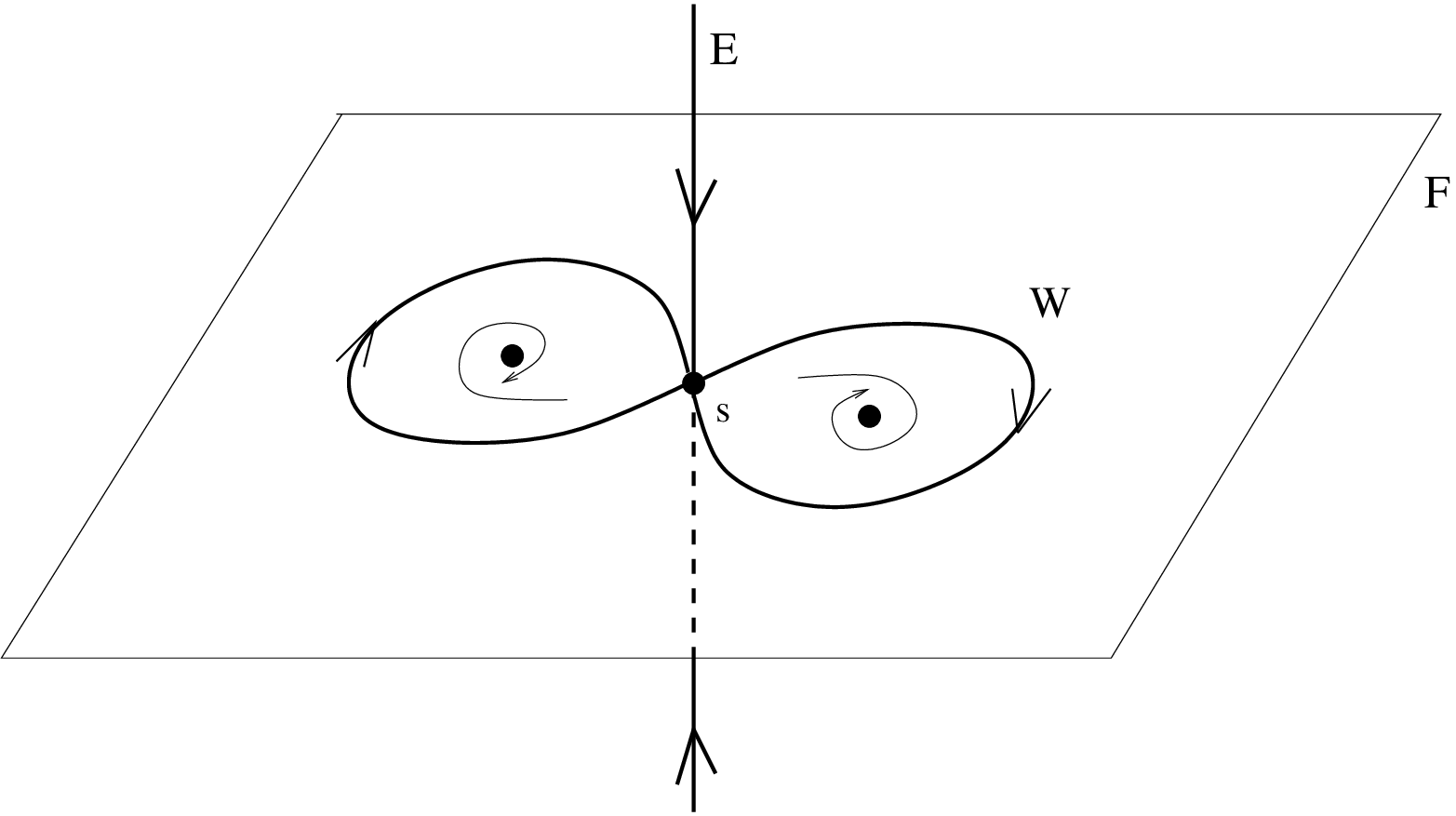}
    \caption{A double homoclinic saddle connection expanding area}
    \label{fig:double-homocl-saddle}
  \end{figure}

  The compact invariant set $W$, for the vector field $X$,
  formed by the pair of homoclinic connections together with
  the singularity $s$ admits a splitting $E\oplus F$, where
  $E$ is the vertical direction and $F$ is the plane
  direction. This splitting is continuous, $E$ is
  uniformly contracting and $F$ is uniformly area expanding.

  Indeed, since we assume that the singularity $s$ expands
  area, i.e., $|\det DX_t(x)\vert_{L_x}|\ge c e^{\lambda t}$ for all
  $t>0$ and all $2$-plane $L_x \subset F_x$, for some fixed constants $c,\lambda>0$, we can fix
  a neighborhood $U$ of $s$ such that
  \begin{align*}
    |\det DX_t(x)\vert_{L_x}|\ge c e^{\lambda t}, \forall 0\le t\le1,
    \forall x\in U
  \end{align*}
  and consequently
  \begin{align*}
    |\det DX_t(x)\vert_{L_x}|\ge c e^{\lambda t}, \forall t>0
  \end{align*}
  whenever the positive orbit of $x$ is contained in
  $U$. Now we observe that there exists $T>0$ such that for
  all $x\in W\setminus U$ we have $X_t(x)\in U$ for all
  $t\in\R$ satisfying $|t|>T$. Hence we can choose $k=\inf\{
  c^{-1}e^{-\lambda t}|\det DX_t(x)\vert_{L_x}|: x\in W\setminus U,
  0\le t\le T\}>0$
  and write $
    |\det DX_t(x)\vert_{L_x}|\ge k\cdot c e^{\lambda t}$
    for $0\le t\le T;$
  and for $t>T$
  \begin{align*}
  |\det DX_t(x)\vert_{L_x}|&=
    |\det DX_{t-T}(X_T(x))\vert_{L_{X_T(x)}}|\cdot |\det DX_T(x)\vert_{L_x}|
    \ge
    c e^{\lambda (t-T)}\cdot k\cdot c e^{\lambda T}
    =
    kc^2 e^{\lambda t}.
  \end{align*}
  So we can find a constant $\kappa>0$ so that $|\det
  DX_t(x)\vert_{L_x}|\ge \kappa e^{\lambda t}$ for all $x\in W$ and
  every $t\ge0$.
  
  To obtain a \emph{non-dominated} splitting, just choose the
  contraction rate along $E$ to be weaker than the
  contraction rate of the singularity $s$.

\begin{example}
  \label{ex:transitive}
  Considering, in the previous Item $(1)$, $\Lambda$ as the
  union of the saddle $s$ with only one homoclinic
  connection, we have $T_\Lambda M=E\oplus F$ a
  non-dominated continuous splitting over a compact
  invariant and isolated set containing a singularity of the
  flow. Moreover, $\Lambda$ is the $\alpha$-limit set of all
  points of the plane $\R^2\times\{0\}$ in the region
  bounded by the curve $\Lambda$ except the sink; see
  \cite[Chapter 4, Section 9]{perko01}.
\end{example}

\vspace{0.2in}

{\bf Item $(2)$:}
  \label{ex:vitor}
 Consider the flow known as ``Bowen example'';
  see e.g. \cite{Ta95} and Figure~\ref{fig:bowen}.
  \begin{figure}[htpb]
    \centering
    \includegraphics[width=11cm]{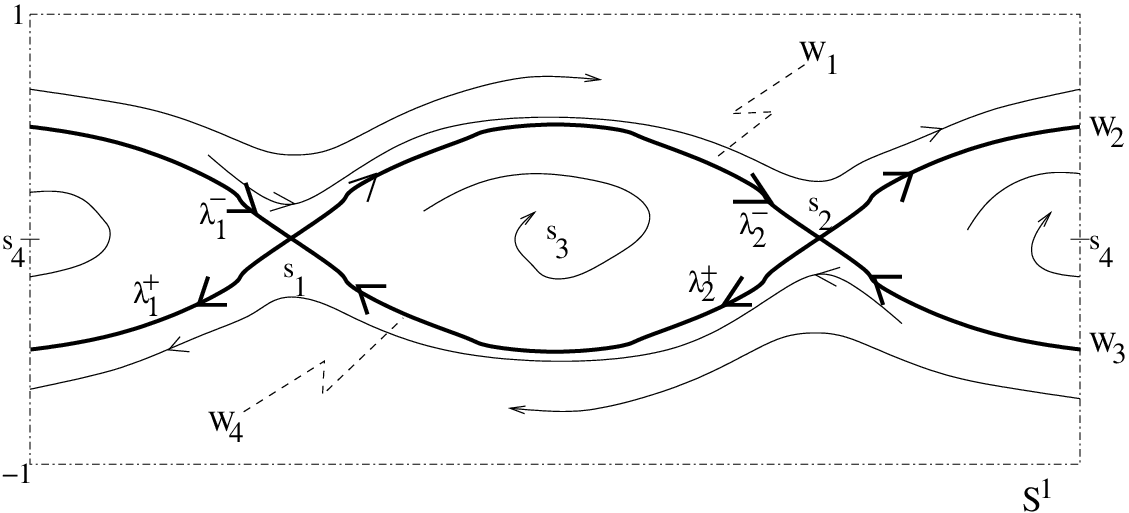}
    \caption{A sketch of Bowen's example flow.}
    \label{fig:bowen}
  \end{figure}
  We have chosen to reverse time with respect to the flow
  studied in \cite{Ta95} so that the heteroclinic connection
  $W=\{s_1,s_2\}\cup W_1\cup W_2 \cup W_3\cup W_4$ is now a
  repeller. The past orbit under this flow $\phi_t$ of every
  point $z\in S^1\times [-1,1] = M$ not in $W$ accumulates
  on either side of the heteroclinic connection, as
  suggested in the figure, if we impose the condition
  $\la_1^{-} \la_2^{-} < \la_1^{+} \la_2^{+}$ on the
  eigenvalues of the saddle fixed points $s_1$ and $s_2$
  (for more specifics on this see \cite{Ta95} and references
  therein).  In this setting, the saddle singularities
  $s_1,s_2$ are area expanding: $|\det D\phi_t(s_i)|$ grows
  exponentially with $t>0$, $i=1,2$.

  We can now embed this system in the $2$-sphere putting two
  sinks at the ``north and south poles'' of $\sS^2$; see
  Figure~\ref{fig:bowen-flow-embedd}.  Let us denote the
  corresponding vector field on $\sS^2$ by $X$.

  \begin{figure}[htpb]
    \centering
    \psfrag{s}{$s_1$} \psfrag{r}{$s_2$}
        \psfrag{t}{$s_3$} \psfrag{u}{$s_4$}
     \includegraphics[width=5cm]{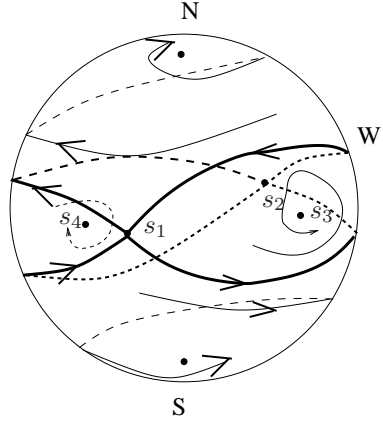}
    \caption{The Bowen flow embedded in a sphere}
    \label{fig:bowen-flow-embedd}
  \end{figure}

  Let $Y$ be a vector field in $\sS^1$ corresponding to the
  ``north-south'' system: the corresponding flow has only
  two fixed points $N$ and $S$, $N$ is a source and $S$ a
  sink.  We can choose the absolute value of the eigenvalues
  of $Y$ at $N,S$ to be between the absolute value of the
  eigenvalues of $s_1,s_2$.

  In this way, for the flow $\{Z_t\}_{t\in\R}$
  associated to the vector field
  $X\times Y$ on $\sS^2\times\sS^1$, the singularities
  $\{s_i\times S\}_{i=1,2}$ of the invariant set
  $\Lambda:=W\times\{S\}$ are hyperbolic saddles, and the
  splitting $T_\Lambda(\sS^2\times\sS^1)=E\oplus F$ given by
  \begin{align*}
    E_{(p,S)}=\{0\}\times T_S\sS^1
    \quad\text{and}\quad
    F_{(p,S)}=T_p\sS^2\times\{0\}
    \quad\text{for}\quad
    p\in W,
  \end{align*}
  is continuous, $DZ_t$-invariant, $E$ is uniformly contracted and
  $F$ is area expanding (by an argument similar to the
  previous example), but it is not dominated.

  Example~\ref{ex:transitive} and items $(1)$ and $(2)$ of
  Proposition \ref{t.examples} show that
  Theorem~\ref{mthm:domination} is false without the
  dominating condition on the splitting $E\oplus F$ at the
  set of singularities within $\Lambda$.

\vspace{0.2in}

{\bf Item $(3)$:}
  \label{ex:no-domination}
  Now we present a suspension flow whose base map has a
  dominated splitting but the flow does not admit any
  dominated splitting.

  Let $f:\T^4\times\T^4$ be the diffeomorphism described
  in \cite{BoV00} which admits a continuous dominated
  splitting $E^{cs}\oplus E^{cu}$ on $\T^4$, but does not
  admit any hyperbolic (uniformly contracting or expanding)
  subbundle.
  \begin{figure}[htpb]
    \includegraphics[height=2cm]{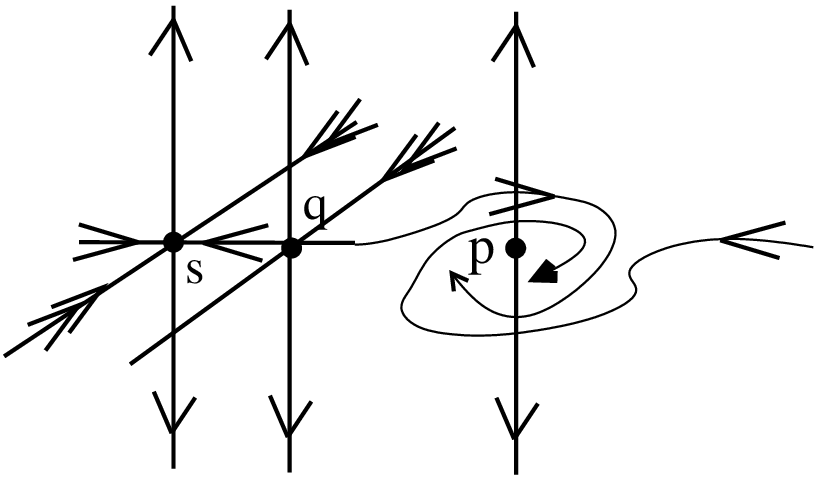}
    \caption{Saddles with real and complex eigenvalues.}
    \label{fig:nonhypcomplex}
  \end{figure}
  There are hyperbolic fixed points of $f$ satisfying (see
  Figure~\ref{fig:nonhypcomplex}):
  \begin{itemize}
  \item $\dim E^u(p)=2=\dim E^s(p)$ and there exists no
    invariant one-dimensional sub-bundle of $E^u(p)$;
  \item $\dim E^u(\widetilde p)=2=\dim E^s(\widetilde p)$ and there exists no
    invariant one-dimensional sub-bundle of $E^s(\widetilde
    p)$;
  \item $\dim E^s(\widetilde q)=3$ and $\dim E^u(q)=3$.
  \end{itemize}
  Hence, the suspension flow of $f$ with constant roof
  function $1$ does not admit any dominated splitting. In
  fact, the natural invariant splitting $E^{cs}\oplus
  E^X\oplus E^{cu}$ is the continuous invariant splitting
  over $\T^4\times [0,1]$ with bundles of least dimension,
  and is not dominated since at the point $p$ the flow
  direction $E^X(p)$ dominates the $E^{cs}(p)=E^s(p)$
  direction, but at the point $q$ this domination is
  impossible.
  
This completes the proof of Proposition~\ref{t.examples}.

  \subsection{Proof of Proposition~\ref{diffeo-case}}

  We present first some auxiliary notions.

Let $f: M \to M$ be a $C^1$-diffeomorphism and $\Gamma
\subset M$ a compact $f$-invariant set, that is,
$f(\Gamma)=\Gamma$.  We say that $\Gamma$ is \emph{transitive} if
there exists $x \in \Gamma$ with dense orbit: the closure
of $\{f^n(x):n\ge1\}$ equals $\Gamma$. The invariant set
$\Gamma$ is \emph{topologically mixing} if, for each pair
$U,V$ of non-empty open sets of $\Gamma$, there exists
$N=N(U,V)\in\Z^+$ such that $U\cap f^n(V)\neq\emptyset$ for
all $n>N$.

We say that $\Gamma$ is an attracting set if there exists a
neighborhood $U$ of $\Gamma$ in $M$ such that $ \cap_{n \geq
  0} f^n(U) = \Gamma$; in this case we say that $U$ is an
isolating neighborhood for $\Gamma$. An \emph{attractor} is
a transitive attracting set and a \emph{repeller} is an
attractor for the inverse diffeomorphism $f^{-1}$.  It is
well-known that hyperbolic attractors or repellers can be
decomposed into finitely many compact subsets which are
permuted and each of these pieces is topologically mixing
for a power of the original map; see
e.g.~\cite{robinson2004}.

Now, we present the example stated in
Proposition~\ref{diffeo-case}.  Consider a hyperbolic
attractor (a Plykin attractor, see \cite{Pl74}) $\Gamma$,
with an isolating neighborhood $U$, defined on the
two-dimensional sphere $\sS^2$ for a diffeomorphism $f:
\sS^2 \to \sS^2$, with a splitting $T_\Gamma\sS^2=E\oplus F$
satisfying $\|Df\mid_E\|\le\lambda_s$ and
$\|Df\mid_F\|\ge\lambda_u$ where $0<\lambda_s < 1 <
\lambda_u$ and $\lambda_s\cdot\lambda_u<1$.

 Now we consider the following diffeomorphism
$$
g = f_1 \times f_2 :M\to M, \quad\text{with}\quad M= \sS^2 \times \sS^2,
$$
where $f_1 := f^2$ and $f_2:= f^{-1}$. We note that
$\Gamma_1 := \cap_{n \geq 0} f_1^n(U)$ is a hyperbolic
attractor with respect to $f_1$ whose contraction and
expansion rates along its hyperbolic splitting
$T_{\Gamma_1}\sS^2=E_1\oplus F_1$ are $\lambda^2_s,
\lambda^2_u$, respectively. Likewise $\Gamma_2:=\cap_{n\le0}
f_2^n(U)$ is a hyperbolic repeller with respect to $f_2$
whose contraction and expansion rates along its hyperbolic
splitting $T_{\Gamma_2}\sS^2=E_2\oplus F_2$ are
$\lambda^{-1}_u, \lambda^{-1}_s$, respectively. This implies
in particular that $Df_1$ contracts area near $\Gamma_1$,
that is, $\lambda_s^2\cdot\lambda_u^2<1$
and that $Df_2$ expands area near $\Gamma_2$,
that is, $(\lambda_s\cdot\lambda_u)^{-1}>1$.
Moreover, the set
$$
\Gamma := \Gamma_1 \times \Gamma_2
$$
is a hyperbolic set for $g$, whose hyperbolic splitting is
given by $(E_1\times E_2)\oplus(F_1\times F_2)$.

Now we note, on the one hand, that the splitting $T_\Gamma
M=E\oplus F$ with $E=E_1$ and $F=F_1\oplus E_2\oplus F_2$ is
not dominated, since the contraction rate along $E_2$ is
stronger than the contraction rate along $E_1$, but $E$ is
uniformly contracted and $F$ is uniformly
sectionally-expanded, since the Lyapunov
  exponents $2\log\lambda_u, -\log\lambda_u, -\log\lambda_s$
  are such that each pair has positive sum.

On the other hand, the decomposition $T_{\Gamma}M =
T_{\Gamma_1}\sS^2 \times T_{\Gamma_2}\sS^2=E\oplus F$ where
$E=E_1\oplus F_1$ and $F=E_2\oplus F_2$ is continuous, since
each subbundle is continuous; is also $Dg$-invariant; but it
is not a dominated decomposition.  In fact, we have the
contraction/expansion rates $0<\lambda^2_s<1< \lambda^2_u$
associated to $E$ and the corresponding rates
$\lambda_u^{-1}, \lambda_s^{-1}$ associated to $F$,
satisfying the relation $\lambda^2_s < \lambda^{-1}_u < 1 <
\lambda^{-1}_s <\lambda^2_u$ (since
$0<\lambda_s\cdot\lambda_u<1$ and $0<\lambda_s<1<\lambda_u$
imply that $\lambda_s^2<\lambda_u^{-1}$ and
$\lambda_s^{-1}<\lambda_u^2$).  Therefore the splitting
$E\oplus F$ cannot be dominated.  In addition, as noted
above, both $Dg^{-1}\mid_E$ and $Dg\mid_F$ expand area, so
that $E$ is sectionally-contracted and $F$ is
secionally-expanded.

Finally, since both $f_1\mid_{\Gamma_1}$ and $f_2\mid_{\Gamma_2}$ are
topologically mixing, then $g\mid_{\Gamma}$ is transitive;
see e.g.~\cite{Wa82}. 

This example shows that there are transitive hyperbolic sets
with a non-dominated, although continuous, splitting
satisfying the sectional-expansion and uniform contraction
or sectional contraction conditions. This completes the
proof of Proposition~\ref{diffeo-case}.

\section*{Acknowledgments}
This is a part of the PhD thesis of L. Salgado at Instituto de
Matem\'atica-Universidade Federal do Rio de Janeiro under a
CNPq (Brazil) scholarship and she is now supported by
INCTMat-CAPES post-doctoral scholarship at IMPA.

The authors thank C. Morales and L. D\'iaz for many
observations that helped improve the statements of the
results.


\def\cprime{$'$}

\end{document}